\documentclass[11pt, twoside]{article}

\usepackage{amssymb}
\usepackage{amsmath}
\usepackage{amsthm}
\usepackage{color}
\usepackage{mathrsfs}

\allowdisplaybreaks

\def\rr{{\mathbb R}}
\def\rn{{\mathbb{R}^n}}

\def\zz{{\mathbb Z}}

\def\nn{{\mathbb N}}

\def\cs{{\mathcal S}}

\def\fz{\infty }
\def\az{\alpha}
\def\bz{\beta}

\def\gz{{\gamma}}

\def\vz{\varphi}

\def\lf{\left}
\def\r{\right}

\def\ls{\lesssim}

\def\tr{\triangle}

\def\wz{\widetilde}
\def\wh{\widehat}

\def\loc{{\mathop\mathrm{\,loc\,}}}
\def\supp{\mathop\mathrm{\,supp\,}}

\def\Xint#1{\mathchoice
{\XXint\displaystyle\textstyle{#1}}%
{\XXint\textstyle\scriptstyle{#1}}%
{\XXint\scriptstyle\scriptscriptstyle{#1}}%
{\XXint\scriptscriptstyle\scriptscriptstyle{#1}}%
\!\int}
\def\XXint#1#2#3{{\setbox0=\hbox{$#1{#2#3}{\int}$ }
\vcenter{\hbox{$#2#3$ }}\kern-.6\wd0}}

\def\dashint{\Xint-}

\def\F{\dot{F}_{p,q}^\a}

\def\ga{\gamma}
\def\be{\b}

\def\F{\dot{F}_{p,q}^\az}
\def\B{\dot{B}_{p,q}^\az}

 \def\tr{{\triangle}}

\def\f{\frac}
\def\vi{\varphi}
\def\({\left(}
\def \){ \right)}

 \def\a{{\alpha}}
 \def\b{{\beta}}

 \def\t{{\theta}}

 \def\d{{\delta}}

 \def\CC{{\mathbb C}}

 \def\NN{{\mathbb N}}

 \def\RR{{\mathbb R}}
  
 \def\ZZ{{\mathbb Z}}

 \def\supp{\operatorname{supp}}

\newcommand{\we}{\wedge}

\newtheorem{theorem}{Theorem}[section]
\newtheorem{lemma}[theorem]{Lemma}

\theoremstyle{definition}
\newtheorem{remark}[theorem]{Remark}
\newtheorem{definition}[theorem]{Definition}
\renewcommand{\appendix}{\par
   \setcounter{section}{0}%
   \setcounter{subsection}{0}%
   \setcounter{subsubsection}{0}%
   \gdef\thesection{\@Alph\c@section}%
   \gdef\thesubsection{\@Alph\c@section.\@arabic\c@subsection}%
   \gdef\theHsection{\@Alph\c@section.}%
   \gdef\theHsubsection{\@Alph\c@section.\@arabic\c@subsection}%
   \csname appendixmore\endcsname
 }

\numberwithin{equation}{section}

\begin{document}

\arraycolsep=1pt

\title{\bf\Large Characterizations of  Besov and Triebel-Lizorkin Spaces via Averages on Balls
\footnotetext{\hspace{-0.35cm} 2010 {\it
Mathematics Subject Classification}. Primary 46E35;
Secondary 42B25, 42B35.
\endgraf {\it Key words and phrases.}  Besov space, Triebel-Lizorkin space,
average on ball, difference, Calder\'on reproducing formula.
\endgraf This project is supported by the National
Natural Science Foundation of China
(Grant Nos.~11171027, 11361020, 11411130053 and 11471042),
the Specialized Research Fund for the Doctoral Program of Higher Education
of China (Grant No. 20120003110003) and the Fundamental Research
Funds for Central Universities of China (Grant Nos.~2013YB60
and 2014KJJCA10). The research of A. Gogatishvili was also
partially supported by the grant P201/13/14743S of the Grant agency of the Czech Republic and RVO: 67985840.}}
\author{Feng Dai, Amiran Gogatishvili, Dachun Yang
and Wen Yuan\,\footnote{Corresponding author}}
\date{}
\maketitle

\vspace{-0.7cm}

\begin{center}
\begin{minipage}{13cm}
{\small {\bf Abstract}\quad
Let $\ell\in\mathbb{N}$
and $p\in(1,\infty]$. In this article, the authors prove that the sequence
$\{f-B_{\ell,2^{-k}}f\}_{k\in\mathbb{Z}}$
consisting of the differences between $f$ and the ball average $B_{\ell,2^{-k}}f$
characterizes the Besov space $\dot B^\alpha_{p,q}(\rn)$ with $q\in (0, \infty]$
and the Triebel-Lizorkin space $\dot F^\alpha_{p,q}(\rn)$ with $q\in (1,\infty]$
when the smoothness order $\alpha\in(0,2\ell)$.
More precisely, it is proved that $f-B_{\ell,2^{-k}}f$ plays the same role as the approximation to the
identity $\varphi_{2^{-k}}\ast f$ appearing in the definitions
of $\dot B^\alpha_{p,q}(\rn)$ and $\dot F^\alpha_{p,q}(\rn)$.
The corresponding results for inhomogeneous Besov and Triebel-Lizorkin spaces
are also obtained. These results, for the first time,
give a way to introduce Besov and Triebel-Lizorkin
spaces with any smoothness order in $(0, 2\ell)$
on spaces of homogeneous type, where $\ell\in{\mathbb N}$.
}
\end{minipage}
\end{center}

\vspace{0.2cm}

\section{Introduction}\label{s1}
\hskip\parindent
It is well known that the theory of function spaces with smoothness
is a central topic of the analysis on spaces of homogenous type in the sense
of Coifman and Weiss \cite{cw71, cw77}.
Recall that the first order Sobolev space on spaces of homogenous type
was originally introduced by Haj\l asz in \cite{h03} and later
Shanmugalingam \cite{sh00} introduced another kind of a first order
Sobolev space which has strong locality and hence is more suitable for
problems related to partial differential equations on spaces
of homogeneous type. Recently, Alabern et al. \cite{amv}
gave a way to introduce Sobolev spaces of any order bigger than $1$
on spaces of homogeneous type in spirit closer to the square function
and Dai et al. \cite{dgyy} gave several other ways, different from \cite{amv}, to
introduce Sobolev spaces of order $2\ell$
on spaces of homogeneous type in spirit closer to the pointwise
characterization as in \cite{h03}, where $\ell\in\nn:=\{1,2,\ldots\}$.
Later, motivated by \cite{amv},
Yang et al. \cite{yyz} gave a way to introduce Besov and Triebel-Lizorkin
spaces with smoothness order in $(0, 2)$
on spaces of homogeneous type. It is still an open question
how to introduce Besov and Triebel-Lizorkin
spaces with smoothness order not less than $2$
on spaces of homogeneous type.

In this article, we establish a characterization of Besov and Triebel-Lizorkin
spaces which can have any positive smoothness order on $\rn$ via the difference
between functions themselves and their ball
averages. Since the average operator used in this article is also well defined on
spaces of homogeneous type, this characterization can be used to introduce
Besov and Triebel-Lizorkin spaces with any positive smoothness order on any
space of homogeneous type and hence our results give an answer
to the above open question.

Let us now give a detailed description of the main ideas used in this article.
It is well known that a locally integrable function $f$ belongs to the Sobolev space
$W^{\az,p}(\rn)$, with $\az\in(0,1)$ and $p\in(1,\fz)$, if and only if $f\in L^p(\rn)$ and
$$s_\az(f):=\lf\{\int_0^\fz\lf[\dashint_{B(\cdot,\,t)}
|f(\cdot)-f(y)|\,dy \r]^2\,\frac{dt}{t^{1+2\az}}\r\}^{1/2}\in L^p(\rn)$$
(see, for example, \cite{w72,s71,t83,y03}).
Here and hereafter, $B(x,t)$ denotes an open ball with center at
$x\in \rn$ and radius  $t\in(0,\fz)$, and
$\dashint_{B(x,t)} f(y) \,dy$ denotes the \emph{integral average} of
$f\in L^1_\loc(\rn)$ on the ball $B(x,t)\subset \rn$, namely,
\begin{equation}\label{av-o-1}
\dashint_{B(x,t)} f(y) \,dy:= \frac1{|B(x,t)|}\int_{B(x,t)} f(y) \,dy=:B_tf(x).
\end{equation}
However, when $\az\in[1,\fz)$, $s_\az(f)$ is not able to characterize $W^{\az,p}(\rn)$, since, in
this case, $f\in L^1_\loc(\rn)$ and $\|s_\az(f)\|_{L^p(\rn)}<\fz$
imply that $f$ must be a constant function (see \cite[Section 4]{gkz} for more details).

Recently,  Alabern et al. \cite{amv} established a remarkable
characterization of Sobolev spaces of smooth order bigger than 1 and they proved that
a function $f\in W^{\az,p}(\rn)$, with $\az\in(0,2)$ and
$p\in(1,\fz)$, if and only if $f\in L^p(\rn)$ and the square function
$S_\az(f)\in L^p(\rn)$, where
$$S_\az(f)(\cdot):=\lf\{\int_0^\fz\lf|\dashint_{B(\cdot,\,t)}[f(\cdot)-f(y)]\,dy \r|^2\,
\frac{dt}{t^{1+2\az}}\r\}^{1/2},\quad f\in L^1_\loc(\rn)$$
(see \cite[Theorem 1 and p.\,591]{amv}).
Comparing $S_\az$ and $s_\az$, we see that the only difference exists in
that the absolute value  $|f(\cdot)-f(y)|$ in $s_\az(f)$ is replaced by
$f(\cdot)-f(y)$ in $S_\az(f)$. However, this slight change induces a  quite different
behaviors between $s_\az(f)$ and $S_\az$ when  characterizing Sobolev spaces.
The former characterizes Sobolev spaces only with smoothness order less than $1$, while
the later characterizes Sobolev spaces with smoothness order less than $2$.
Such a difference follows from the following observation:
for all $f\in C^2(\rn)$ and $t\in(0,\,1)$,
\begin{equation}\label{key}
\dashint_{B(x,\,t)}[f(x)-f(y)]\,dy= O(t^2),\quad x\in\rn,
\end{equation}
which follows from the Taylor expansion of $f$ up to order $2$:
$$f(y)=f(x)+\nabla f(x)\cdot(y-x)+ O(|y-x|^2),\quad x,\,y\in\rn;$$
in other words, the $S_\az$-function provides
smoothness up to order $2$. We point out that this phenomenon was first
observed by Wheeden in \cite{w69}
(see also \cite{w72}), and later independently
by Alabern, Mateu and Verdera \cite{amv}.

By means of the fact \eqref{key}, Alabern et al. \cite[Theorems 2 and 3]{amv}
also characterized Sobolev spaces of higher smoothness order and showed that
$f\in W^{\az,p}(\rn)$, with $\az\in [2N,2N+2)$, $N\in\nn$ and $p\in(1,\fz)$,
if and only if $f\in L^p(\rn)$ and there exist functions $g_1$, $\ldots$, $g_N\in L^p(\rn)$
such that $S_\az(f,g_1,\ldots,g_N)\in L^p(\rn)$, where
$$S_\az(f,g_1,\ldots,g_N)(\cdot):=\lf\{\int_0^\fz \lf|\dashint_{B(\cdot,t)} t^{-\az} R_N(y,\cdot)\,dy\r|^2\,\frac{dt}t\r\}^{1/2}$$
with
\begin{eqnarray}\label{e1.1}
R_N(y;\,\cdot)&&:= f(y)-f(\cdot)-\sum_{j=1}^{N}g_j(\cdot)|y-\cdot|^{2j}
\end{eqnarray}
when $\az\in(2N,2N+2)$, and
\begin{eqnarray}\label{e1.2}
R_N(y;\,\cdot)&&:= f(y)-f(\cdot)-\sum_{j=1}^{N-1}g_j(\cdot)|y-\cdot|^{2j}-B_tg_N(\cdot)|y-\cdot|^{2N}
\end{eqnarray}
when $\az=2N$. Indeed, the function $g_j$ was proved in
\cite[Theorems 2 and 3]{amv} to equal to $\frac1{L_j}\Delta^jf$
almost everywhere, where $L_j:=\Delta^j|x|^{2j}$ for $j\in\{1,\,\ldots,\,N\}$.
As the corresponding results for Triebel-Lizorkin spaces,
Yang et al. \cite[Theorems, 1.1, 1.3 and 4.1]{yyz} further proved that, for all
$\az\in(2N,2N+2)$, $N\in\nn$
and $p\in (1,\fz]$, the Besov space $\dot B^{\az}_{p,q}(\rn)$ with $q\in (0,\fz]$
and the Triebel-Lizorkin space $\dot F^{\az}_{p,q}(\rn)$ with $q\in (1,\fz]$
can be characterized via the function
\begin{equation}\label{e1.3}
 S_{\az,\,q}(f)(x):=\lf\{\sum_{k\in\zz} 2^{k\az q}
 \lf|\dashint_{B(x,\,2^{-k})}\wz R_N(y;\,x)\,dy\r|^q \r\}^{1/q}, \quad x\in\rn,
\end{equation}
where, for all $x,\,y\in\rn$ and $k\in\zz$,
\begin{eqnarray}\label{e1.4}
\wz R_N(y;\,x):= f(y)-f(x)-\sum_{j=1}^{N}\frac 1{L_j}\Delta^jf(x)|y-x|^{2j}.
\end{eqnarray}
It is an open question, posed in \cite[Remark 4.1]{yyz}, whether this exists a
corresponding characterization for $\dot B^{\az}_{p,q}(\rn)$ and $\dot F^{\az}_{p,q}(\rn)$
when $\az=2N$ with $N\in\nn$. Moreover, only when $\az\in (0,2)$,
\cite[Theorems 1.1 and 4.1]{yyz} provide a way to introduce Besov and
Triebel-Lizorkin spaces with smoothness order $\az$ on spaces of
homogeneous type.

Via higher order differences,
Triebel \cite{t10,t11} and Haroske and Triebel
\cite{ht11,ht13} obtained another characterization of Sobolev spaces with order
bigger than $1$ on $\rn$ without involving derivatives.
Recall that, for $\ell\in\nn$, the \emph{$\ell$-th order (forward)
difference operator} $\wz\Delta^\ell_h$ with $h\in\rn$ is defined by setting,
for all functions $f$ and $x\in\rn$,
$$\wz\Delta^1_hf(x):=f(x+h)-f(x),\qquad \wz\Delta^\ell_h:=
\wz\Delta^{1}_h\wz\Delta^{\ell-1}_h,\quad \ell\ge 2.$$
By means of  $\wz\Delta^\ell_hf$, Triebel \cite{t10,t11} and Haroske and Triebel
\cite{ht11,ht13} proved that
the Sobolev space $W^{\ell,p}(\rn)$ with $\ell\in\nn$
and $p\in(1,\fz)$ can be characterized by a pointwise inequality in the spirit
of Haj\l asz \cite{h96} (see also Hu \cite{h03} and Yang \cite{y03}). Recall that the difference
$\wz\Delta^\ell_hf$ can also be used to characterize Besov spaces and
Triebel-Lizorkin spaces with smoothness order no more than $\ell$.
We refer the reader to Triebel's monograph \cite[Section 3.4]{t92}
for these difference characterizations of Besov and Triebel-Lizorkin spaces;
see also \cite[Section 3.1]{rs}.
However, it is still unclear how to define higher than $1$ order differences on
spaces of homogeneous type.

On the other hand, recall that the averages of a
function $f$ can be used to approximate $f$ itself
in some function spaces; see, for example, \cite{dp,bdd}.
Motivated by \eqref{key} and
the pointwise characterization of Sobolev spaces with smoothness order no more than 1
(see Haj\l asz \cite{h96}, Hu \cite{h03} and Yang \cite{y03}),
the authors established in \cite{dgyy} some pointwise
characterizations of Sobolev spaces
with smoothness order $2\ell$ on $\rn$ via ball averages
of $f$, where $\ell\in\nn$. To be precise,
as the higher order variants of $B_t$ in \eqref{av-o-1},
for all $\ell\in\nn$, $t\in(0,\fz)$ and $x\in\rn$, we define the
\emph{$2\ell$-th order average operator   $B_{\ell,t}$}
by setting, for all $f\in L_\loc^1(\rn)$ and $x\in\rn,$
\begin{equation}\label{1:edu}
B_{\ell,t}f(x) := -\frac{2}{\binom{2\ell}{\ell} }
\sum_{j=1}^\ell (-1)^j \binom {2\ell}{\ell-j} B_{jt} f(x),
\end{equation}
here and hereafter, $\binom {2\ell}{\ell-j}$ denotes the binomial coefficients.
Obviously, $B_{1,t}f=B_tf$. Moreover, it was observed in \cite{dgyy} that
$f-B_{\ell, t} f$ is a $2\ell$-th order central
difference of the   function $t \mapsto B_t f(x)$ with step $t$ at the origin,
namely, for all $\ell\in\nn$, $t\in(0,\fz)$, $f\in L_\loc^1(\rn)$ and $x\in\rn$,
\begin{equation}\label{diff}
f(x)-B_{\ell, t} f(x) =   \f {(-1)^{\ell}  }{\binom {2\ell} \ell}   \tr_t ^{2\ell} g (0)
\end{equation}
with
\begin{equation}\label{g-d}
g(t):=\begin{cases}
B_t f(x),\quad &t\in(0,\fz); \\
f(x),\quad &t=0;\\
B_{-t}f(x),\quad &t\in(-\fz,0).
\end{cases}
\end{equation}
Here and hereafter, for all functions $h$ on $\mathbb{R}$
and $\theta,\, t\in \rr$, let $T_\t h(t) := h(t+\t)$, and
the \emph{central difference operators $\tr_t^r$} are defined by setting
\begin{align*}
\tr_\t^1h(t)&:=\tr_\t h (t):=h\lf(t+\f \t 2\r) -h\lf(t-\f \t2\r)=\lf(T_{\t/2}  -T_{-\t/2}\r)h(t),\\
   \tr_\t^r h (t) &:= \tr_\t (\tr_\t^{r-1} h)(t) =\sum_{j=0}^r \binom{r}
   j (-1)^j h\lf( t + \f {r\t}2-j\t\r),\quad r\in\{2,3,\ldots\}.
\end{align*}
The authors proved in \cite{dgyy} that $f\in W^{2\ell,p}(\rn)$, with $\ell\in\nn$ and $p\in (1,\fz)$, if and only if
$f\in L^p(\rn)$ and  there exist a non-negative $g\in L^p(\rn)$ and a positive constant $C$ such that $|f(x)-B_{\ell,t}f(x)|\le Ct^{2\ell}\,g(x)$
for all $t\in(0,\fz)$ and almost every $x\in \rn$. Various variants of this pointwise
characterization were also presented in \cite{dgyy}. Recall that centered averages or their combinations were used
to measure the smoothness and to characterize the $K$-functionals in \cite{dd04,dr99,di93}.

Comparing the difference  $f-B_{\ell,t}f$ with the
usual difference $\wz\Delta^{2\ell}_hf$, we find that
the former has an advantage that it
involves  only  averages of $f$ over balls, and hence
can be easily generalized to any space of homogeneous type,
whereas the difference operator  $\wz\Delta^{2\ell}_hf$ can not.
We can also see their difference via \eqref{diff}. Indeed, it follows from
\eqref{diff} that
$f-B_{\ell, t} f$ is a $2\ell$-th order central difference of a function $g$ and the
parameter related to such a difference is the radius $t\in(0,\fz)$ of the ball $B(x,t)$ with $x\in\rn$,
while the parameter related to $\wz\Delta^{2\ell}_hf$ is $h\in \rn$, which also curbs
the extension of $\wz\Delta^{2\ell}_hf$ to spaces of homogeneous type.

Although there exist differences between $f-B_{\ell,t}f$ and
the usual difference $\wz\Delta^{2\ell}_hf$,
the characterizations of $W^{2\ell,p}(\rn)$ via $f-B_{\ell,t}$ obtained
in \cite{dgyy} imply that, in some sense, $f-B_{\ell,t}f$ also plays
the role of $2\ell$-order derivatives. Therefore, it is natural to ask
\emph{whether we can use  $f-B_{\ell,t}$ to characterize
Besov and Triebel-Lizorkin spaces
with smoothness order less than $2\ell$ or not}.

The main purpose of this article is to answer this question.
To this end, we first recall some basic notions.
Let $\zz_+:=\nn\cup\{0\}$ and $\mathcal{S}(\rn)$ denote the collection of all
\emph{Schwartz functions} on $\rn$, endowed
with the usual topology, and $\mathcal{S}'(\rn)$ its \emph{topological dual}, namely,
the collection of all bounded linear functionals on $\mathcal{S}(\rn)$
endowed with the weak $\ast$-topology.
Let $\cs_\fz(\rn)$ be the set of all Schwartz functions $\vz$ such that
$\int_\rn x^\gz \vz(x)\,dx=0$ for all $\gz\in \zz_+^n$, and
$\cs'_\fz(\rn)$ its topological dual.
For all $\az\in\zz_+^n$, $m\in\zz_+$ and $\vz\in\cs(\rn)$, let
$$\|\vz\|_{\az,m}:=\sup_{x\in\rn,\,|\bz|\le |\az|}(1+|x|)^m |\partial^\bz \vz(x)|.$$

For all $\vz\in\cs'_\fz(\rn)$,
we use $\widehat\vz$ to denote its \emph{Fourier transform}.
For any $\vz\in \cs(\rn)$ and $t\in(0,\fz)$, we let $\vz_t(\cdot):=t^{-n}\vz(\cdot/t)$.

For all $a\in\rr$, $\lfloor a\rfloor $ denotes the \emph{maximal integer} no more than $a$.
For any $E\subset\rn$, let $\chi_E$ be its \emph{characteristic function}.

We now recall the notions of
Besov and Triebel-Lizorkin spaces; see \cite{t83,t92,FJ90,ysy}.

\begin{definition}\label{d2.1}
Let $\az\in(0,\,\fz)$, $p,\,q\in(0,\,\fz]$ and $\vz\in\cs(\rn)$ satisfy that
\begin{equation}\label{con}
{\rm supp}\,\widehat\vz\subset \{\xi\in\rn:\ 1/2\le|\xi|\le2\}\ \mathrm{and}\
|\widehat\vz(\xi)|\ge {\rm constant}>0 \ \mathrm{if}\ 3/5\le|\xi|\le5/3.
\end{equation}

(i) The {\it homogenous Besov space} $
\dot B^\az_{p,\,q}(\rn)$ is defined as the collection of all
$f\in\cs'_\fz(\rn)$ such that $\|f\|_{\dot B^\az_{p,\,q}(\rn)}<\fz$, where
\begin{equation*}
\|f\|_{   \dot B^\az_{p,\,q}(\rn)}:=\lf[\sum_{k\in\zz}2^{k\az q}
  \|\vz_{2^{-k}}\ast
f\|_{L^p(\rn)}^q\r]^{1/q}
\end{equation*}
with the usual modifications made when $p=\fz$ or $q=\fz$.

(ii) The {\it homogenous Triebel-Lizorkin space} $
\dot F^\az_{p,\,q}(\rn)$ is defined as the collection of all
$f\in\cs'_\fz(\rn)$ such that $\|f\|_{  \dot F^\az_{p,\,q}(\rn)}<\fz$,
where, when $p\in(0,\fz)$,
\begin{equation*}
\|f\|_{\dot F^\az_{p,\,q}(\rn)}:=\lf\|\lf[\sum_{k\in\zz}2^{k\az q}
  |\vz_{2^{-k}}\ast
f|^q\r]^{1/q}\r\|_{L^p(\rn)}
\end{equation*}
with the usual modification made when $q=\fz$,
and
\begin{equation*}
\|f\|_{ \dot F^\az_{\fz,\,q}(\rn)}:=\sup_{x\in\rn}
\sup_{m\in\zz}\lf\{\dashint_{B(x,\,2^{-m})}\sum_{k=m}^\fz2^{k\az q}
 |\vz_{2^{-k}}\ast
f(y)|^q\,dy\r\}^{1/q}
\end{equation*}
with the usual modification made when $q=\fz$.
\end{definition}

It is well known that the spaces
$\dot B^\az_{p,\,q}(\rn)$ and
$\dot F^\az_{p,\,q}(\rn)$ are independent
of the choice of functions $\vz$ satisfying \eqref{con};
see, for example, \cite{fjw91}.

We also recall the corresponding inhomogeneous spaces.

\begin{definition}\label{d2.2}
Let $\az\in(0,\,\fz)$, $p,\,q\in(0,\,\fz]$, $\vz\in\cs(\rn)$ satisfy \eqref{con}
and $\Phi\in \cs(\rn)$ satisfy that
\begin{equation}\label{con2}
{\rm supp}\,\widehat\Phi\subset \{\xi\in\rn:\ |\xi|\le2\}\ \mathrm{and}\
|\widehat\Phi(\xi)|\ge {\rm constant}>0 \ \mathrm{if}\ |\xi|\le5/3,
\end{equation}

(i) The {\it inhomogeneous Besov space} $
B^\az_{p,\,q}(\rn)$ is defined as the collection of all
$f\in\cs'(\rn)$ such that $\|f\|_{B^\az_{p,\,q}(\rn)}<\fz$, where
\begin{equation*}
\|f\|_{B^\az_{p,\,q}(\rn)}:=\lf[\sum_{k\in\zz_+} 2^{k\az q}
  \|\vz_{2^{-k}}\ast
f\|_{L^p(\rn)}^q\r]^{1/q}
\end{equation*}
with the usual modifications made when $p=\fz$ or $q=\fz$, where, when $k=0$,
$\vz_{2^{-k}}$ is replaced by $\Phi$.

(ii) The {\it inhomogeneous Triebel-Lizorkin space} $
F^\az_{p,\,q}(\rn)$ is defined as the collection of all
$f\in\cs'(\rn)$ such that $\|f\|_{F^\az_{p,\,q}(\rn)}<\fz$,
where, when $p\in(0,\fz)$,
\begin{equation*}
\|f\|_{F^\az_{p,\,q}(\rn)}:=\lf\|\lf[\sum_{k\in\zz_+}2^{k\az q}
  |\vz_{2^{-k}}\ast
f|^q\r]^{1/q}\r\|_{L^p(\rn)}
\end{equation*}
with the usual modification made when $q=\fz$,
and
\begin{equation*}
\|f\|_{F^\az_{\fz,\,q}(\rn)}:=\sup_{x\in\rn}
\sup_{m\in\zz_+}\lf\{\dashint_{B(x,\,2^{-m})}\sum_{k=m}^\fz2^{k\az q}
 |\vz_{2^{-k}}\ast
f(y)|^q\,dy\r\}^{1/q}
\end{equation*}
with the usual modification made when $q=\fz$,  where, when $k=0$,
$\vz_{2^{-k}}$ is replaced by $\Phi$.
\end{definition}

It is also well known that the spaces $
B^\az_{p,\,q}(\rn)$ and $
F^\az_{p,\,q}(\rn)$ are independent of the choice of functions $\vz$ and $\Phi$ satisfying
\eqref{con} and \eqref{con2}, respectively;
see, for example, \cite{t83}.

As the main result of this article, we prove that the difference
$f-B_{\ell,2^{-k}}f$ with $k\in\zz$ plays the same role of the approximation to
the identity $\vz_{2^{-k}}\ast f$ in the definitions of Besov and Triebel-Lizorkin
spaces in the following sense.

\begin{theorem}\label{t-bf}
Let $\ell\in\nn$ and $\a\in (0, 2\ell)$.

{\rm (i)} Let $p\in(1,\fz]$ and $q\in(0,\infty]$.
If $f\in \B(\rn)$, then there exists
$g\in L_\loc^1(\rn)\cap\cs'_\fz(\rn)$ such that $g=f$ in $\cs'_\fz(\rn)$
and $|||g|||_{\B(\rn)}\le C\|f\|_{\B(\rn)}$
for some positive constant $C$ independent of $f$,
where
$$|||g|||_{\B(\rn)}:=\left\{\sum_{k\in\zz}
2^{k\az q}\|g-B_{\ell,2^{-k}}g\|_{L^p(\rn)}^q\right\}^{1/q}.$$

Conversely, if $f\in L_\loc^1(\rn)\cap\cs'_\fz(\rn)$ and
$|||f|||_{\B(\rn)}<\fz$, then $f\in \B(\rn)$ and $\|f\|_{\B(\rn)}
\le C|||f|||_{\B(\rn)}$ for some positive constant $C$ independent of $f$.

{\rm (ii)}  Let $p\in(1,\fz]$ and $q\in(1,\infty]$. If $f\in \F(\rn)$, then there exists
$g\in L_\loc^1(\rn)\cap\cs'_\fz(\rn)$ such that $g=f$ in $\cs'_\fz(\rn)$
and $|||g|||_{\F(\rn)}\le C\|f\|_{\F(\rn)}$
for some positive constant $C$ independent of $f$, where, when $p\in(1,\fz)$,
$$|||g|||_{\F(\rn)}:=\left\|\lf\{\sum_{k\in\zz}
2^{k\az q}|g-B_{\ell,2^{-k}}g|^q\right\}^{1/q}\right\|_{L^p(\rn)}$$
and, when $p=\fz$,
$$|||g|||_{\dot{F}^\az_{\fz,q}(\rn)}:=\sup_{x\in\rn}
\sup_{m\in\zz}\lf\{\dashint_{B(x,\,2^{-m})}\sum_{k=m}^\fz2^{k\az q}
 |g(y)-B_{\ell,2^{-k}}g(y)|^q\,dy\r\}^{1/q}.$$

Conversely, if $f\in L_\loc^1(\rn)\cap\cs'_\fz(\rn)$ and
$|||f|||_{\F(\rn)}<\fz$, then $f\in \F(\rn)$ and $\|f\|_{\F(\rn)}
\le C|||f|||_{\F(\rn)}$ for some positive constants $C$ independent of $f$.
\end{theorem}

\begin{remark}
(i) Notice that $f-B_{\ell,2^{-k}}f$ can be easily defined on
any space of homogeneous type. Thus, the characterizations of $\B(\rn)$ and
$\F(\rn)$ obtained in Theorem \ref{t-bf} provide a possible way
to introduce Besov and Triebel-Lizorkin spaces with arbitrary positive
smoothness order on spaces of homogeneous type, while the characterizations
of Besov and Triebel-Lizorkin spaces via the usual differences can not.

(ii) Observing that $f-B_{1,t}f=f-B_tf$, we see that, when $\az\in(0,2)$,
Theorem \ref{t-bf} just coincides with \cite[Theorems 1.1 and 4.1]{yyz}.
When $\az\in(2,\fz)$, comparing with Theorem \ref{t-bf} and
\cite[Theorems 1.2 and 4.1]{yyz}, we find that the former provides
a way to introduce Besov and Triebel-Lizorkin spaces with smoothness order no less than $2$ on spaces of
homogeneous type, while the later has a restriction that $\az$ can not be
any even positive integer and also can not be generalized to any space of homogeneous type,
due to the lack of derivatives on spaces of homogeneous type.

(iii) In \cite{hmy}, a concept of RD-spaces was introduced,
namely, a space of homogeneous type whose measure also satisfies
the inverse doubling condition is called an \emph{RD-space}
(see also \cite{hmy, yz11} for several equivalent definitions
of RD-spaces). Via approximations to the identity,
a theory of Besov and Triebel-Lizorkin spaces with smoothness order
in $(-1, 1)$ on RD-spaces was also systematically developed
in \cite{hmy}. Then, a natural and interesting question is: on RD-spaces, whether
the Besov and Triebel-Lizorkin spaces with smoothness order in $(0,1)$ defined in \cite{hmy}
coincide with those defined via $f-B_{1,2^{-k}}f$ in sprit of Theorem \ref{t-bf}
or not. We will not seek an answer of this question in this article.
\end{remark}

Theorem \ref{t-bf} is proved in Section \ref{s2}. Comparing with
those proofs for various pointwise characterizations of Sobolev spaces
$W^{2\ell,p}(\rn)$ via $f-B_{\ell,t}f$ in \cite{dgyy}, the proof of Theorem \ref{t-bf}
 is much more complicated. Indeed, the main idea of the proof for
Theorem \ref{t-bf} is to write $f-B_{\ell,2^{-k}}f$ as a convolution operator,
then control $f-B_{\ell,2^{-k}}f$ by
certain maximal functions via calculating pointwise estimates of the related
operator kernel and finally apply the vector-valued maximal
inequality of Fefferman and Stein in \cite{fs71}. The Calder\'on
reproducing formula on $\rn$ (see, for example, \cite{fjw91})
also plays a key role in this proof.

In Section \ref{s3}, we further show that the inhomogeneous variant of Theorem \ref{t-bf}
also holds true (see Theorem \ref{t-bf-i} below). We also show that Theorems \ref{t-bf}
and \ref{t-bf-i} still hold true on Euclidean spaces with non-Euclidean metrics.

Finally, we make some conventions on notation.
The \emph{symbol}  $C $ denotes   a positive constant
which depends
only on the fixed parameters $n,\,\az,\,p,\,q$ and possibly on auxiliary functions,
unless otherwise stated; its value  may vary from line to line.
We use the \emph{symbol} $A \ls B$ to denote that
there exists a positive constant $C$ such that
 $A \le C \,B$.
The symbol $A \sim B$ is used as an abbreviation of
$A \ls B \ls A$. We also use the \emph{symbol} $\lfloor s\rfloor$
for any $s\in\rr$ to denote the maximal integer not more than $s$.

\section{Proof of Theorem \ref{t-bf}}\label{s2}
\hskip\parindent
To prove Theorem \ref{t-bf}, we need some technical lemmas.
Let, for all $t\in(0,\fz)$ and $x\in\rn$,
$I(x):=\f 1 {|B(0,1)|} \chi_{B(0,1)} (x)$ and  $I_t(x):=t^{-n} I(x/t)$.
Then
 $$( B_{\ell, t} f)(x) =\f {-2} {\binom{2\ell}{\ell}} \sum_{j=1}^\ell (-1)^j \binom{2\ell}{\ell-j}
 ( f * I_{jt})(x),\quad x\in\rn,\ \ t\in(0,\fz),$$
and hence
\begin{equation}\label{mq}
(B_{\ell, t} f)^{\we} (\xi) = m_\ell (t\xi) \wh{f} (\xi),\quad    \  \xi\in\rn,
\end{equation}
where
\begin{equation}\label{2-1}
m_\ell (x):= \f {-2}{\binom{2\ell}{\ell}} \sum_{j=1}^\ell (-1)^j \binom{ 2\ell} { \ell-j} \wh{I} (j x),\quad    \    x\in\rn.
\end{equation}
A straightforward calculation shows that
\begin{equation}\label{2-2}
 \wh{I}(x) =\ga_n\int_0^1 \cos (u|x|) (1-u^2) ^{\f {n-1}2} \, du, \    \  x\in\rn,
 \end{equation}
with
$\ga_n := [ \int_0^1(1-u^2) ^{\f {n-1}2} \, du]^{-1}$
(see also Stein's book \cite[p.\,430, Section 6.19]{s93}).

\begin{lemma}\label{lem-2-1}
For all $\ell\in\nn$ and $x\in\rn$,
\begin{equation}\label{2-3}
m_{\ell}(x) =1-A_\ell(|x|),\end{equation}
where
\begin{equation}\label{2-4}
A_\ell(s) := \ga_n \f {4^{\ell}}{\binom {2\ell}{\ell}} \int_0^1 (1-u^2)^{\f
{n-1}2} \lf(\sin \f {us}2\r)^{2\ell} \, du,\   \   \  s\in\mathbb{R}.\end{equation}
Furthermore, $s^{-2\ell} A_{\ell}(s)$ is a   smooth function on $\mathbb{R}$ satisfying that
there exist positive constants $c_1$ and $c_2$ such that
\begin{equation}\label{2-5}
0<c_1 \leq \f {A_{\ell} (s)}{s^{2\ell}} \leq c_2,\    \    \ s\in(0,4]
\end{equation}
and
$$\sup_{s\in \mathbb{R}}
\lf| \lf(\f {d} {ds}\r)^i \lf(\f {A_{\ell} (s)}{s^{2\ell}}\r)\r| <\infty,\
\   \  i\in\nn.$$
\end{lemma}

\begin{proof}
Combining \eqref{2-1}  with \eqref{2-2}, we obtain
 \begin{equation}\label{2-6}
 m_\ell (x) =\f {-2\ga_n} {\binom{2\ell}{\ell}}
 \int_0^1  \lf[\sum_{j=1}^\ell (-1)^j \binom {2\ell} {\ell-j}\cos ( j u |x|) \r](1- u^2)^{\f {n-1}2} \, du,\quad x\in\rn.
 \end{equation}
 However, a straightforward calculation shows that, for all $s\in \mathbb{R}$,
  $$4^{\ell} \lf(\sin \f s2\r)^{2\ell} = \binom {2\ell} {\ell}  +
  2\sum_{j=1}^\ell (-1)^j \binom {2\ell}{\ell-j} \cos j s.$$
  This, together with \eqref{2-6}, implies \eqref{2-3}.

Next we show \eqref{2-5}. By the mean value theorem, we know that,
for all $u\in(0,1)$ and $s\in\rr$, there exists $\theta\in(0,1)$ such that
$$\left(\sin\frac{us}{2}\right)^{2\ell}=
\lf(\frac12us\r)^{2\ell}\left(\cos \frac{us\theta}2\right)^{2\ell}.$$
From this and \eqref{2-4}, we deduce that, for all $s\in(0,4]$,
$$\f {A_{\ell} (s)}{s^{2\ell}}\le \ga_n \f {4^{\ell}}{2\binom {2\ell}{\ell}} \int_0^1 (1-u^2)^{\f
  {n-1}2} u^{2\ell} \, du=:c_2<\fz$$
and
\begin{eqnarray*}
\f {A_{\ell} (s)}{s^{2\ell}}
&&\ge \ga_n \f {4^{\ell}}{2\binom {2\ell}{\ell}} \int_0^{\min\{1,\frac{2\pi}{3s}\}} (1-u^2)^{\f
  {n-1}2} u^{2\ell} \left(\cos \frac{us\theta}2\right)^{2\ell}\, du\\
&&\ge \ga_n \f {1}{2\binom {2\ell}{\ell}} \int_0^{\min\{1,\frac{2\pi}{3s}\}} (1-u^2)^{\f
  {n-1}2} u^{2\ell} \, du\\
  &&\ge \gamma_n \frac1{2\binom{2\ell}{\ell}}\int_0^{\frac{\pi}6}(1-u^2)^{\frac{n-1}2}u^{2\ell}\,du=:c_1>0.
\end{eqnarray*}
These prove \eqref{2-5}.

Finally, by the mean value theorem again, an argument similar to the above
also implies that  $$\sup_{s\in \mathbb{R}}
  \lf| \lf(\f {d} {ds}\r)^i \lf(\f {A_{\ell} (s)}{s^{2\ell}}\r)\r|<\fz$$
  for all $i\in\nn$. This finishes the proof of Lemma \ref{lem-2-1}.
  \end{proof}

Recall that the \emph{Hardy-Littlewood maximal operator $M$} is defined by
setting, for all $f\in L^1_\loc(\rn)$,
 $$Mf(x):=\sup_{B\subset \rn} \dashint_B |f(y)|\,dy,\quad x\in\rn,$$
 where the supremum is taken over all balls $B$ in $\rn$ containing $x$.
The following two lemmas can be verified straightforwardly.

\begin{lemma}\label{lem-2-2}  Let  $\{T_t\}_{t\in(0,\fz)}$ be a family of multiplier
operators given  by setting, for all $f\in L^2(\rn)$,
 $$ ( T_t f)^{\we} (\xi) := m(t\xi) \wh{f}(\xi),\quad     \     \xi\in \rn,\   \   t\in(0,\fz)$$
 for some $ m\in L^\infty(\rn)$. If
 $$ \|\nabla^{n+1}  m \|_{L^1(\rn)}+\|m\|_{L^1(\rn)}  \leq C_1<\infty, $$
 then there exists a positive constant $C$ such that, for all $f\in L^2(\rn)$ and $x\in\rn$,
 $$ \sup_{t\in(0,\fz)} | T_t f(x)|\le  C C_1\     Mf(x).$$
\end{lemma}

\begin{proof}
For all $t\in(0,\fz)$, $f\in L^2(\rn)$ and $x\in\rn$,
by the Fubini theorem, we see that
\begin{eqnarray*}
|T_tf(x)|&&=\lf|\int_\rn m(t\xi)\widehat{f}(\xi) e^{ix\cdot\xi}\,d\xi\r|\\
&&=\lf|\int_\rn f(y)\int_\rn m(t\xi) e^{i(x-y)\cdot\xi}\,d\xi\,dy\r|\\
&&\le\lf|\int_{|x-y|<t} f(y)\int_\rn m(t\xi) e^{i(x-y)\cdot\xi}\,d\xi\,dy\r|+\lf|\int_{|x-y|\ge t}\cdots\r|=:{\rm I}+{\rm II}.
\end{eqnarray*}

It is easy to see that ${\rm I}\ls \|m\|_{L^1(\rn)}Mf(x)$.

For ${\rm II}$, via the Fubini theorem and the integration by parts,
we also have
\begin{eqnarray*}
{\rm II}&&\ls \int_{|x-y|\ge t} \frac{|f(y)|}{|x-y|^{n+1}}\int_\rn t^{n+1}|\nabla^{n+1}m(t\xi)| \,d\xi\,dy\\
&&\ls \|\nabla^{n+1}  m \|_{L^1(\rn)}\sum_{j=1}^\fz t\int_{2^{jt}\le|x-y|<2^{j+1}t}
\frac{|f(y)|}{|x-y|^{n+1}}\,dy\\
&&\ls \|\nabla^{n+1}  m \|_{L^1(\rn)}\sum_{j=1}^\fz 2^{-j}Mf(x)
\ls \|\nabla^{n+1}  m \|_{L^1(\rn)}Mf(x),
\end{eqnarray*}
which completes the proof of Lemma \ref{lem-2-2}.
\end{proof}

\begin{remark}\label{rem-e}
(i) The above proof of Lemma \ref{lem-2-2} actually yields the following more subtle estimate,
which is also needed in the proof of Theorem \ref{t-bf}: Assume that $f\in L^2(\rn)$, $x\in B(z,s)$ for some $z\in\rn$
and $s\in(0,\fz)$. Then
there exists a positive constant $C$, independent of $t$, $s$, $f$ and $x$, such that, for all
$l\in \nn\cap (n,\fz)$ and $t\in(0,\fz)$,
\begin{eqnarray*}
|T_tf(x)|&&\le C(\|m\|_{L^1(\rn)}+ \|\nabla^{l}m\|_{L^1(\rn)})
\sum_{i=0}^\fz 2^{-i(l-n)}M(f\chi_{B(z,2^it+s)})(x).
\end{eqnarray*}
Indeed, notice that $x\in B(z,s)$, $y\in B(x,t)$ and $t\in(0,\fz)$ imply that
$y\in B(z, t+s)$. Thus, by the Fubini theorem, we find that
\begin{eqnarray*}
{\rm I}\le \lf|\int_{|x-y|<t} f(y)\chi_{B(z,t+s)}(y)\int_\rn m(t\xi) e^{i(x-y)\cdot\xi}\,d\xi\,dy\r|\ls \|m\|_{L^1(\rn)}M(f\chi_{B(z,t+s)})(x).
\end{eqnarray*}
Similarly, by the Fubini theorem and the integration by parts, together with $\ell\in \nn\cap (n,\fz)$,
we know that
\begin{eqnarray*}
{\rm II}&&\ls \int_{|x-y|\ge t} \frac{|f(y)|}{|x-y|^{l}}\int_\rn t^{l}|\nabla^{l}m(t\xi)| \,d\xi\,dy\\
&&\ls \|\nabla^{l} m \|_{L^1(\rn)} \int_{|x-y|\ge t} \frac{t^{l-n}|f(y)|}{|x-y|^{l}}\,dy\\
&&\ls \|\nabla^{l}m\|_{L^1(\rn)}\sum_{i=1}^\fz (2^it)^{-l} t^{l-n}
\int_{|x-y|\sim 2^it}|f(y)|\chi_{B(z,2^it+s)}(y)\,dy\\
&&\ls \|\nabla^{n+1}m\|_{L^1(\rn)}\sum_{i=1}^\fz 2^{-i(l-n)}M(f\chi_{B(z,2^it+s)})(x),
\end{eqnarray*}
where $|x-y|\sim 2^it$ means $2^{i-1}t\le |x-y|< 2^it$. This finishes the proof of the above claim.

(ii) We also point out that the conclusions of Lemma \ref{lem-2-2} and (i) of this remark
remain true for all $f\in L^p(\rn)$ with $p\in(1,\infty)$ and $m\in \cs(\rn)$.
\end{remark}

From the H\"older inequality when $q\in [1,\fz]$ and the monotonicity
of $l^q$ when $q\in (0,1)$, we immediately deduce the following conclusions,
the details being omitted.

\begin{lemma}\label{lem-2-3}
Let $\{ a_j \}_{j\in \ZZ}\subset \CC$,  $q\in(0,\fz]$ and $\b\in(0,\fz)$.
Then there exists a positive constant $C$, independent of $\{a_j\}_{j\in\zz}$, such that
\begin{equation*} 
\left[\sum_{k\in \ZZ}  2^{k\be  q}\left
( \sum_{j=k}^\infty |a_j|\right)^q \right]^{1/q} \leq C \left(\sum_{k\in\ZZ} 2^{k \be q} |a_k| ^q  \right)^{1/q}
\end{equation*}
and
\begin{equation*}
\left[\sum_{k\in\ZZ} 2^{-k \b q}\left( \sum _{ j=-\infty}^k |a_j|\right)^q \right]^{1/q} \leq C \left(\sum_{k\in\ZZ} 2^{ - k\b q}|a_k|^q \right)^{1/q}.
\end{equation*}
\end{lemma}

Now we prove Theorem \ref{t-bf}.

\begin{proof}[Proof of Theorem \ref{t-bf}]
We only prove (ii), the proof of (i) being similar and easier.

To show (ii), let $\vz\in\cs(\rn)$ satisfy \eqref{con} and
$\sum_{j\in\zz} \widehat{\vz}_j\equiv 1$ on $\rn\setminus\{0\}$. Assume first that $\az\in(0,2\ell)$,
$p\in(1,\fz)$ and $q\in(1,\fz]$.
Let $f\in \dot F^\az_{p,\,q}(\rn)$.
We know that
$\dot F^\az_{p,\,q}(\rn) \hookrightarrow L^1_\loc(\rn)$
in the sense of distributions;
see, for example, \cite[Propsition 4.2]{my},
\cite[Proposition 5.1]{yz11} or \cite[Proposition 8.2]{ysy} for a proof.
Indeed, it was proved therein that there exists a sequence $\{P_j\}_{j\in\zz}$ of polynomials
of degree not more than $\lfloor \az-n/p\rfloor $ such that the summation $\sum_{j\in\zz}(\vz_j\ast f+P_j)$
converges in $L^1_\loc(\rn)$ and $\cs'_\fz(\rn)$ to a function $g\in L^1_\loc(\rn)$, which is known to be the
Calder\'on reproducing formula (see, for example, \cite{FJ90,fjw91}).
The function $g$ serves as a representative of $f$.
Thus, in the below proof, we identify $f$ with $g$.
Then $g\in L^1_\loc(\rn)\cap \cs'_\fz(\rn)$.
Now we show $|||g|||_{\F(\rn)}\ls \|f\|_{\F(\rn)}$, namely,
 \begin{equation}\label{3-1}
 \left\|\lf\{\sum_{k\in\zz}
2^{k\az q}|g-B_{\ell,2^{-k}}g|^q\right\}^{1/q}\right\|_{L^p(\rn)}\ls \|f\|_{\F(\rn)}.
 \end{equation}
To this end,
for all $k,\,j\in\zz$ and $\xi\in\rn\setminus\{0\}$, define
$T_{k,j}$ as
 \begin{equation}\label{3-2}
 (T_{k, j} f )^{\we} (\xi):=\wh{\vi} (2^{-j} \xi) A_{\ell} ( 2^{-k } |\xi|)  \wh{f}(\xi),\quad \xi\in\rn.
 \end{equation}
Noticing that the degree of each $P_j$ is not more than $\lfloor \az-n/p\rfloor<2\ell$
and $P-B_{\ell,2^{-k}}P=0$ for all polynomials $P$ of degree less than $2\ell$,
we then find that
\begin{equation} \label{3-3}
g- B_{\ell, 2^{-k}} g = \sum_{j\in \ZZ}
T_{k, j} f.
\end{equation}
We split the sum $\sum_{j\in\ZZ}$ in this last equation into two parts $\sum_{j\ge k}$ and $\sum_{j<k}$. The first part is relatively easy
to deal with. Indeed, for $j\ge k$, by \eqref{3-2}, we see that, for all $x\in\rn$,
\begin{align}\label{in}
  |T_{k,j} f(x)|& = |( I- B_{\ell, 2^{-k}} ) ( f\ast \vi_{2^{-j}}) (x)| \notag\\
  &\leq |f \ast \vi_{2^{-j}}(x)| +C_{\ell} \sum _{ i=1}^\ell |B_{i 2^{-k}} ( f\ast \vi_{2^{-j}} )(x)|  \notag\\
  &\ls M (f\ast \vi_{2^{-j}})(x).
 \end{align}
From this and Lemma \ref{lem-2-3}, it follows that
\begin{align}\label{3-4}
\sum_{k\in\ZZ}  2^{k\a q}\lf|\sum_{ j\ge k} T_{k,j}   f\r|^q&\ls  \sum_{k\in\ZZ}  2^{k\a q}\lf[\sum_{ j\ge k}    M (f\ast \vi_{2^{-j}}) \r]^q\notag\\
&\ls \sum_{j\in\ZZ} 2^{j q\a} [ M (f\ast \vi_{2^{-j}})]^q.
\end{align}

Now we  handle the sum $\sum_{j<k}$. Since $\vz$ satisfies \eqref{con},
by \cite[Lemma (6.9)]{fjw91}, there exists $\psi\in \cs(\rn)$ satisfying
\eqref{con} such that
$$\sum_{j\in\zz} \wh\vz(2^{-j}\xi)\wh\psi(2^{-j}\xi)=1,\quad \xi\in\rn\setminus\{0\}.$$
Thus, for all $\xi\in \rn\setminus\{0\}$,
$$(T_{k,j} f)^{\we} (\xi) = \wh{\vi} (2^{-j} \xi) A_{\ell }(2^{-k} |\xi|) \wh{f_j} (\xi)=m_{k,j}(\xi) \wh{f}_j(\xi),$$
  where $f_j :=\sum_{i=-1}^1f\ast \psi_{2^{i-j}}$ and
  $$ m_{k,j} (\xi) :=  \wh{\vi} (2^{-j} \xi) \f {A_{\ell }(2^{-k} |\xi|) } { ( 2^{-k } |\xi|)^{2\ell}  } ( 2^{-k } |\xi|)^{2\ell},\qquad \xi\in\rn\setminus\{0\}.$$
Write $\wz m_{k,j}(\xi):= m_{k,j} (2^j\xi)$. From  Lemma \ref{lem-2-1}, it follows that,
for all $j<k$ and $\xi\in \rn\setminus\{0\}$,
\begin{equation}\label{e-m-0}
 |\partial^\bz \wz m_{k,j} (\xi)| \ls  2^{ 2\ell (j-k)} \chi_{\overline{B(0,2)}\setminus B(0,1/2)}(\xi), \qquad    \bz\in\zz_+^d,
\end{equation}
and hence $\|\wz m_{k,j}\|_{L^1(\rn)}
+\|\nabla^{n+1}\wz m_{k,j}\|_{L^1(\rn)}\ls  2^{ 2\ell (j-k)}$,
which, together with Lemma \ref{lem-2-2}, implies that
$$|T_{k,j} f(x)| \ls 2^{ 2\ell (j-k)} M f_j (x),\qquad x\in\rn.$$
Thus,
by Lemma \ref{lem-2-3},  for $\a\in (0, 2\ell)$, we have
\begin{align}\label{3-5}
\sum_{k\in\ZZ}2^{ k\a q}  \lf|\sum_{j=-\infty }^k  T_{k,j} f\r|^q
&\ls \sum_{k\in\ZZ}2^{k(\a-2\ell)q}
\lf( \sum_{j=-\infty}^k 2^{2\ell j} M f_j \r)^q \notag\\
&\ls \sum_{j\in \ZZ}2^{j\a q} [M f_j]^q .
\end{align}

Combining \eqref{3-4} and \eqref{3-5}  with \eqref{3-3}, and using the Fefferman-Stein
vector-valued maximal inequality (see \cite{fs71} or \cite{s93}), we see that
\begin{eqnarray*}
&&\left\|\lf\{\sum_{k\in\zz}
2^{k\az q}|g-B_{\ell,2^{-k}}g|^q\right\}^{1/q}\right\|_{L^p(\rn)}\\
&&\quad\ls  \left\|\lf\{\sum_{k\in\zz}
2^{k\az q}[M(f\ast \vz_{2^{-k}})]^q\right\}^{1/q}\right\|_{L^p(\rn)}\ls \|f\|_{\F(\rn)}.
 \end{eqnarray*}
This proves \eqref{3-1} and hence finishes
the proof of the first part of Theorem \ref{t-bf}(ii).

To see the inverse conclusion, we only need to prove
\begin{equation}\label{3-6}
\|f\|_{\F(\rn)}\ls \left\|\lf\{\sum_{k\in\zz}
2^{k\az q}|f-B_{\ell,2^{-k}}f|^q\right\}^{1/q}\right\|_{L^p(\rn)}
\end{equation}
whenever $f\in L^1_\loc(\rn)\cap \cs'_\fz(\rn)$ and the right-hand side of \eqref{3-6}
is finite.
To this end, we first claim that
\begin{equation}\label{3-7}
|f\ast \vi_{2^{-j}}(x)| \ls M ( f -B_{\ell, 2^{-j}} f)(x),\qquad j\in\zz,\ \ x\in\rn.
\end{equation}
Indeed, we see that, for all $j\in\zz$ and $\xi\in\rn\setminus\{0\}$,
\begin{align*}
 ( f\ast \vi_{2^{-j}} )^{\we}  (\xi) = \f {\wh{\vi} (2^{-j}\xi) }
 { A_{\ell} ( 2^{-j}|\xi|)} ( f -B_{\ell, 2^{-j}} f)^{\we} (\xi)
 =:\eta(2^{-j}\xi) ( f -B_{\ell, 2^{-j}} f)^{\we} (\xi),
\end{align*}
where $\eta(\xi):=\f { \wh{\vi} (\xi) } { A_{\ell}(|\xi|)}$ for
all $\xi\in\rn\setminus\{0\}$, which is well defined due to \eqref{2-5}.
By Lemma \ref{lem-2-1}, we know that
$\eta\in C_c^\infty(\rn)$ and $\supp\eta \subset \{\xi\in\rn:  \     \
\f12 \leq |\xi|\leq 2\}$.  The claim \eqref{3-7} then follows from Lemma \ref{lem-2-2}.

Now, using the claim \eqref{3-7} and the Fefferman-Stein vector-valued maximal inequality (see \cite{fs71}
or \cite{s93}),
we find that
\begin{align*}
\|f\|_{\F}&\ls     \lf\|\lf\{\sum_{j\in\ZZ} 2^{ j\a q}
\lf [  M ( f-B_{\ell, 2^{-j}} f)\r]^q \r\}^{1/q}\r\|_{L^p(\rn)} \\
&\ls \lf \|\lf[\sum_{j\in \ZZ} 2^{j \a q}
| f-B_{\ell, 2^{-j}} f|^q \r]^{\f 1q} \r\|_{L^p(\rn)}\sim |||f|||_{\F(\rn)}.
\end{align*}
This proves the desired conclusion when $\az\in(0,2\ell)$,
$p\in(1,\fz)$ and $q\in(1,\fz]$.

It remains to consider the case that  $\az\in(0,2\ell)$,
$p=\fz$ and $q\in(1,\fz]$. The proof is similar to that of the case $p\in(1,\fz)$
but more subtle. Assume first that
$f\in \dot{F}^\az_{\fz,q}(\rn)$. By an argument similar to the above, in this case, we need to show
 \begin{equation}\label{3-1x}
\sup_{x\in\rn} \sup_{m\in\zz}\lf\{\dashint_{B(x,\,2^{-m})}\sum_{k=m}^\fz2^{k\az q}
 |g(y)-B_{\ell,2^{-k}}g(y)|^q\,dy\r\}^{1/q}\ls \|f\|_{\dot{F}^\az_{\fz,q}(\rn)}.
 \end{equation}
Notice that, if $y\in B(x,2^{-m})$ and $z\in B(y,i2^{-k})$ with $k\ge m$ and $i\in\{1,\ldots,\ell\}$, then $z\in B(x,(\ell+1)2^{-m})$. Then, similar to
\eqref{in}, we know that, for $j\ge k$ and $y\in B(x,2^{-m})$,
\begin{align*}
  |T_{k,j} f(y)|& = |( I- B_{\ell, 2^{-k}} ) ( f\ast \vi_{2^{-j}}) (y)| \\
  &\leq |f \ast \vi_{2^{-j}}(y)| +C_{\ell} \sum _{ i=1}^\ell |B_{i 2^{-k}} ( f\ast \vi_{2^{-j}} )(y)|  \notag\\
  &\ls M (|f\ast \vi_{2^{-j}}|\chi_{B(x,(\ell+1)2^{-m})})(y),\notag
 \end{align*}
which, together with \eqref{3-3} and Lemma \ref{lem-2-3}, implies that
\begin{align}\label{3-4x}
\sum_{k\ge m}  2^{k\a q}\lf|\sum_{ j\ge k} T_{k,j}   f(y)\r|^q&\ls  \sum_{k\ge m}  2^{k\a q}\lf[\sum_{ j\ge k}    M (|f\ast \vi_{2^{-j}}|\chi_{B(x,(\ell+1)2^{-m})})(y) \r]^q\notag\\
&\ls \sum_{j\ge m} 2^{j \a q}
[ M (|f\ast \vi_{2^{-j}}|\chi_{B(x,(\ell+1)2^{-m})})(y)]^q.
\end{align}

When $m\le j<k$, by \eqref{e-m-0}, and using
Remark \ref{rem-e}(i) instead of Lemma \ref{lem-2-2},
we find that, for all  $k\ge m$, integer $l\ge n+1$ and $y\in B(x,2^{-m})$,
$$|T_{k,j} f(y)| \ls 2^{ 2\ell (j-k)} \sum_{i=0}^\fz 2^{-i(l-n)}
M(f_j\chi_{B(x,2^{i-j}+2^{-m})}) (y)$$
and hence, by Lemma \ref{lem-2-3} and the Minkowski inequality, we see that
\begin{align}\label{3-5x}
&\lf\{\sum_{k\ge m}2^{ k\a q}  \lf|\sum_{j=m}^k  T_{k,j} f(y)\r|^q\r\}^{1/q}\notag\\
&\quad\ls \sum_{i=0}^\fz 2^{-i(l-n)}\lf\{\sum_{k\ge m}2^{k(\a-2\ell)q}
\lf[ \sum_{j=m}^k 2^{2\ell j}
M (f_j\chi_{B(x,(2^{i}+1)2^{-m}})(y) \r]^q \r\}^{1/q}\notag\\
&\quad\ls \sum_{i=0}^\fz 2^{-i(l-n)}\lf\{\sum_{j=m}^\fz 2^{j\a q}
[M (f_j\chi_{B(x,(2^{i}+1)2^{-m}})(y)]^q\r\}^{1/q}.
\end{align}

When $j<m\le k$, we invoke \eqref{e-m-0} and the proof of Remark \ref{rem-e}(i)
to find that, for all $y\in\rn$,
\begin{eqnarray}\label{3-6x}
\qquad |T_{k,j}f(y)|&&\le\lf|\int_{|z-y|<2^{-j}} f_j(z)\int_\rn \wz m_{k,j}(2^{-j}\xi) e^{i(x-y)\cdot\xi}\,d\xi\,dz\r|+\lf|\int_{|z-y|\ge 2^{-j}}\cdots\r|\notag\\
&&\ls 2^{2\ell(j-k)}\dashint_{|z-y|<2^{-j}} |f_j(z)|\,dz \notag\\
&&\quad+\int_{|z-y|\ge 2^{-j}} \frac{|f_j(z)|}{|z-y|^{l}}\int_\rn 2^{-jl}
|\nabla^{l} \wz m_{k,j}(2^{-j}\xi)| \,d\xi\,dz \notag\\
&&\ls 2^{2\ell(j-k)}\sum_{i=0}^\fz 2^{-i(l-n)}
\dashint_{|z-y|\sim 2^{i-j}}|f_j(z)|\,dz \notag\\
&&\ls 2^{2\ell(j-k)}\sum_{i=0}^\fz 2^{-i(l-n)} 2^{-j\az}
\|f\|_{\dot{F}^\az_{\fz,q}(\rn)}\notag\\
&&\ls 2^{2\ell(j-k)}2^{-j\az}
\|f\|_{\dot{F}^\az_{\fz,q}(\rn)},
\end{eqnarray}
where $|z-y|\sim 2^{i-j}$ means that $2^{i-j-1}\le |z-y|<2^{i-j}$ and we chose $l>n$.

Combining \eqref{3-3}, \eqref{3-4x}, \eqref{3-5x} and \eqref{3-6x},
and applying the  Minkowski inequality and the  boundedness of
$M$ on $L^q(\rn)$ with $q\in (1,\fz]$, we know that, for all  $m\in\zz$ and $x\in\rn$,
\begin{eqnarray*}
&&\lf\{\dashint_{B(x,\,2^{-m})}\sum_{k=m}^\fz2^{k\az q}
|g(y)-B_{\ell,2^{-k}}g(y)|^q\,dy\r\}^{1/q}\\
&&\quad\ls \lf\{\dashint_{B(x,\,2^{-m})} \sum_{j\ge m} 2^{j \a q}
[M (|f\ast \vi_{2^{-j}}|\chi_{B(x,(\ell+1)2^{-m})})(y)]^q\,dy\r\}^{1/q}\\
&&\quad\quad+\sum_{i=0}^\fz 2^{-i(l-n)}
\lf\{\dashint_{B(x,\,2^{-m})} \sum_{j\ge m} 2^{j\a q}
[M (f_j\chi_{B(x,(2^{i}+1)2^{-m})})(y)]^q\,dy\r\}^{1/q}\\
&&\quad\quad+\lf\{\dashint_{B(x,\,2^{-m})}
\sum_{k\ge m} 2^{k\az q} \lf[\sum_{j\le m-1} 2^{2\ell(j-k)} 2^{-j\az}
\r]^q\,dy\r\}^{1/q}\|f\|_{\dot{F}^\az_{\fz,q}(\rn)}\\
&&\quad\ls  \lf\{\dashint_{B(x,\,(\ell+1)2^{-m})}  \sum_{j\ge m} 2^{j \a q}
|f\ast \vi_{2^{-j}}(y)| \,dy\r\}^{1/q}\\
&&\quad\quad+\sum_{i=0}^\fz 2^{-i(l-n)}2^{in/q}
\lf\{\dashint_{B(x,(2^{i}+1)2^{-m})} \sum_{j\ge m} 2^{j\a q}
|f_j(y)|^q\,dy\r\}^{1/q}+\|f\|_{\dot{F}^\az_{\fz,q}(\rn)}\\
&&\quad \ls\|f\|_{\dot F^\az_{\fz,q}(\rn)},
\end{eqnarray*}
where we took $l>n(1+1/q)$. This proves \eqref{3-1x}.

Finally, the inverse estimate of \eqref{3-1x} is deduced from an argument similar
to that used in the above proof for \eqref{3-1x}, with $\wz m_{k,j}$ and $f_j$ therein replaced by
$\eta:=\f { \wh{\vi}} { A_{\ell}(|\cdot|)}$ and $f -B_{\ell, 2^{-j}} f$,
respectively. This finishes the proof for the case  $\az\in(0,2\ell)$,
$p=\fz$ and $q\in(1,\fz]$, and hence  Theorem \ref{t-bf}.
 \end{proof}

\section{Inhomogeneous Spaces and Further Remarks}\label{s3}
\hskip\parindent
In this section, we first present the inhomogeneous version of Theorem
\ref{t-bf}. As a further generalization, we show that the conclusions of Theorems
\ref{t-bf} and \ref{t-bf-i} remain valid on Euclidean spaces with non-Euclidean metrics.

It is known that, when $p\in(1,\fz)$ and $\az\in(0,\fz)$, then
$B^\az_{p,\,q}(\rn) \cup F^\az_{p,\,q}(\rn) \subset L^p(\rn)$, while when
 $p=\fz$ and $\az\in(0,\fz)$, then
$B^\az_{\fz,\,q}(\rn) \cup F^\az_{\fz,\,q}(\rn) \subset C(\rn)$, where
$C(\rn)$ denotes the set of all complex-valued  uniformly continuous functions
on $\rn$ equipped with the sup-norm;
see, for example,
\cite[Theorem 3.3.1]{st} and \cite[Chapter 2.4, Corollary 2]{rs}.

\begin{theorem}\label{t-bf-i}
Let $\ell\in\nn$ and $\a\in (0, 2\ell)$.

{\rm (i)} Let $q\in(0,\infty]$.
Then $f\in B^{\az}_{p,q}(\rn)$ if and only if $f\in L^p(\rn)$ when $p\in (1,\fz)$ or $f\in C(\rn)$
when $p=\fz$, and
$$|||f|||_{B^\az_{p,q}(\rn)}:=\|f\|_{L^p(\rn)}+\left\{\sum_{k=1}^\fz
2^{k\az q}\|f-B_{\ell,2^{-k}}f\|_{L^p(\rn)}^q\right\}^{1/q}<\fz.$$
Moreover, $|||\cdot|||_{B^\az_{p,q}(\rn)}$ is equivalent to
$\|f\|_{B^\az_{p,q}(\rn)}$.

{\rm (ii)}  Let $p\in(1,\fz]$ and $q\in(1,\infty]$. Then
$f\in F^\az_{p,q}(\rn)$ if and only if $f\in L^p(\rn)$ when $p\in (1,\fz)$ or $f\in C(\rn)$
when $p=\fz$, and $|||f|||_{F^\az_{p,q}(\rn)}<\fz$,
where, when $p\in(1,\fz)$,
$$|||f|||_{F^\az_{p,q}(\rn)}:=\|f\|_{L^p(\rn)}+\left\|\lf\{\sum_{k=1}^\fz
2^{k\az q}|f-B_{\ell,2^{-k}}f|^q\right\}^{1/q}\right\|_{L^p(\rn)}$$
and, when $p=\fz$,
$$|||f|||_{F^\az_{\fz,q}(\rn)}:=\|f\|_{L^\fz(\rn)}+\sup_{x\in\rn}
\sup_{m\ge 1}\lf\{\dashint_{B(x,\,2^{-m})}\sum_{k=m}^\fz2^{k\az q}
 |f(y)-B_{\ell,2^{-k}}f(y)|^q\,dy\r\}^{1/q}.$$
 Moreover, $|||\cdot|||_{F^\az_{p,q}(\rn)}$ is equivalent to
$\|\cdot\|_{F^\az_{p,q}(\rn)}$.
\end{theorem}

\begin{proof} By similarity, we only consider (ii). The proof is similar to that of Theorem \ref{t-bf}, and we mainly describe the difference.
We need to use the following well-known result: when $\az\in(0,\fz)$ and $p,\,q\in(1,\fz]$, then, for all
$f\in F^\az_{p,q}(\rn)$,
\begin{equation}\label{equi}
\|f\|_{F^\az_{p,q}(\rn)}\sim \|f\|_{L^p(\rn)}+\widetilde{\|f\|}_{F^\az_{p,q}(\rn)},
\end{equation}
where $\widetilde{\|f\|}_{F^\az_{p,q}(\rn)}$ is defined as ${\|f\|}_{F^\az_{p,q}(\rn)}$
in Definition \ref{d2.2} with $k\in\zz_+$ and $m\in\zz_+$ therein replaced,
respectively, by $k\in\nn$ and $m\in\nn$ (which can be easily seen from
\cite[Theorem 3.3.1]{st} and \cite[Chapter 2.4, Corollary 2]{rs}).

Assume first that $f\in F^\az_{p,q}(\rn)$. By \cite[Theorem 3.3.1]{st}
and \cite[Chapter 2.4, Corollary 2]{rs}, we know that $f\in L^p(\rn)$
when $p\in (1,\fz)$ or $f\in C(\rn)$
when $p=\fz$.
On the other hand, repeating the proof of Theorem \ref{t-bf}, we see that, when $p\in(1,\fz)$,
$$\left\|\lf\{\sum_{k=1}^\fz
2^{k\az q}|f-B_{\ell,2^{-k}}f|^q\right\}^{1/q}\right\|_{L^p(\rn)}\ls \|f\|_{F^\az_{p,q}(\rn)}$$
and, when $p=\fz$,
$$\sup_{x\in\rn}
\sup_{m\ge 1}\lf\{\dashint_{B(x,\,2^{-m})}\sum_{k=m}^\fz2^{k\az q}
 |f(y)-B_{\ell,2^{-k}}f(y)|^q\,dy\r\}^{1/q}\ls \|f\|_{F^\az_{\fz,q}(\rn)}$$
which show $|||f|||_{{F^\az_{p,q}(\rn)}}\ls \|f\|_{F^\az_{p,q}(\rn)}$.

Conversely, assume that $f\in L^p(\rn)$
when $p\in (1,\fz)$ or $f\in C(\rn)$
when $p=\fz$, and $|||f|||_{F^\az_{p,q}(\rn)}<\fz$.
Again the proof of Theorem \ref{t-bf} shows that, when $p\in(1,\fz)$,
$$\widetilde{\|f\|}_{F^\az_{p,q}(\rn)}\ls \left\|\lf\{\sum_{k=1}^\fz
2^{k\az q}|f-B_{\ell,2^{-k}}f|^q\right\}^{1/q}\right\|_{L^p(\rn)}$$
and, when $p=\fz$,
$$\widetilde{\|f\|}_{F^\az_{\fz,q}(\rn)}\ls\sup_{x\in\rn}
\sup_{m\ge 1}\lf\{\dashint_{B(x,\,2^{-m})}\sum_{k=m}^\fz2^{k\az q}
 |f(y)-B_{\ell,2^{-k}}f(y)|^q\,dy\r\}^{1/q}\ls \|f\|_{F^\az_{\fz,q}(\rn)}.$$
This, together with \eqref{equi}, further implies that $\|f\|_{F^\az_{\fz,q}(\rn)}
\ls |||f|||_{F^\az_{\fz,q}(\rn)}$, and hence finishes the proof of Theorem \ref{t-bf-i}.
\end{proof}

Finally, we point out that the \emph{conclusions of Theorems \ref{t-bf} and \ref{t-bf-i}
are independent of the choice of the metric in $\rn$}.
To be precise, let $\|\cdot\|$ be a norm in $\rn$, which is not necessarily
the usual Euclidean norm. Then $(\RR^n, \|\cdot\|)$ is
a finite dimensional normed vector space with the unit ball
$$K:=\{x\in \RR^n:\   \ \|x\|\leq 1\}.$$
Clearly, $K$ is  a compact and  symmetric convex set in $\RR^n$ satisfying  that $-K=K$ and $B(0, \d_1)\subset K\subset B(0, \d_2)$ for some $\d_1, \d_2\in(0,\fz)$.

For all $\ell\in\NN$, $f\in L^1_\loc(\rn)$ and  $x\in\rn$, define
$$B_{\ell, t, K} f:=\f {-2}{\binom{2\ell}{\ell}}\sum_{j=1}^\ell (-1)^j \binom{2\ell}{\ell-j} B_{jt, K}f(x).$$
Then, we have the following conclusion.

\begin{theorem}\label{t-g}
The conclusions of  Theorems \ref{t-bf} and \ref{t-bf-i} remain valid
with $B_{\ell, t}$ therein replaced by $B_{\ell, t, K}$.
\end{theorem}

Since the proof of Theorem \ref{t-g} is essentially similar to the proofs
of Theorems \ref{t-bf}  and \ref{t-bf-i}, we only describe the main differences,
the other details being omitted.

We first observe that
$$(B_{\ell, t, K} f)^{\wedge} (\xi):=m_{\ell,K}(t\xi) \wh{f}(\xi),\quad   \  \xi\in\rn,$$
where
$$m_{\ell,K} (x) := \f {-2}{\binom{2\ell}{\ell}}\sum_{j=1}^\ell (-1)^j \binom{2\ell}{\ell-j} \wh{I_K}(jx), \quad x\in\rn.$$
Similar to the proof of Lemma \ref{lem-2-1}, by means of
the symmetry property of $K$,
a straightforward calculation shows that, for all $x\in\rn$,
\begin{align*}
    m_{\ell,K}(x) &=\dashint_K \f {-2}{\binom{2\ell}{\ell}}\sum_{j=1}^\ell (-1)^j \binom{2\ell}{\ell-j} \cos (jx\cdot y)\, dy=:1-A_{\ell,K}(x),
\end{align*}
where
$$A_{\ell,K}(x):=\f {4^{\ell}}{\binom{2\ell}{\ell}}\dashint_K \lf( \sin \f {x\cdot u}2\r)^{2\ell}\, du;$$
Furthermore, we have the following estimates:
for all $x\in\rn$ with $|x|\leq 4$,
\begin{equation}\label{5-1}
0<C_1\leq \f {A_{\ell,K}(x)}{|x|^{2\ell}}\leq C_2,
\end{equation}
and
\begin{equation}\label{5-2}
|\nabla^i A_{\ell,K}(x)|\leq C \min\{|x|^{2\ell-i},1\},\quad   \ i\in\nn,
\end{equation}
where $C_1$, $C_2$ and $C$ are positive constants independent of $x$. Similar to the proof of Lemma \ref{lem-2-2}, by
\eqref{5-1} and \eqref{5-2},  we observe that
$$\sup_{t\in(0,\fz)} \dashint_K |f(x+ty)|\, dy \ls Mf(x)$$
for all $f\in L^1_\loc(\rn)$ and $x\in\rn$.

Finally, notice that, by the equivalence of norms on finite-dimensional vector spaces,
the spaces $\B(\rn)$, $\F(\rn)$ and their inhomogeneous counterparts
are essentially independent of the
choice of the norm of the underlying space $\rn$. By means of this observation and
using \eqref{5-1}, \eqref{5-2} in place of Lemma \ref{lem-2-1},
we obtain Theorem \ref{t-g} via some arguments similar to those used in
the proofs of Theorems \ref{t-bf} and \ref{t-bf-i}, the details being omitted.

\medskip

\noindent{\bf Acknowledgements.} The authors would like to express their
deep thanks to the referee for his very careful reading and so many useful
comments which do improve the presentation of this article.

\bigskip

\noindent  Dai Feng

\medskip

\noindent  Department of Mathematical and Statistical Sciences,
University of Alberta, Edmonton, AB, T6G 2G1, Canada

\smallskip

\noindent {\it E-mail}: \texttt{fdai@ualberta.ca}

\bigskip

\noindent  Amiran Gogatishvili

\medskip

\noindent Institute of Mathematics of the Academy of Sciences of the Czech Republic,
\v Zitn\'a 25, 115 67 Prague 1, Czech Republic

\smallskip

\noindent {\it E-mail}: \texttt{gogatish@math.cas.cz}

\bigskip

\noindent  Dachun Yang and Wen Yuan (Corresponding author)

\medskip

\noindent  School of Mathematical Sciences, Beijing Normal University,
Laboratory of Mathematics and Complex Systems, Ministry of
Education, Beijing 100875, People's Republic of China

\smallskip

\noindent {\it E-mails}: \texttt{dcyang@bnu.edu.cn} (D. Yang)

\hspace{1.1cm}\texttt{wenyuan@bnu.edu.cn} (W. Yuan)


\begin{thebibliography}{100}

\bibitem{amv}  R. Alabern, J. Mateu and J. Verdera,
A new characterization of Sobolev spaces on $\rr^n$,
Math. Ann. 354 (2012), 589-626.


\vspace{-.3cm}

\bibitem{cw71}
R. R. Coifman and G. Weiss, Analyse Harmonique Non-Commutative sur Certains Espaces Homog\`enes,
Lecture Notes in Math. 242, Springer-Verlag, Berlin-New York, 1971.

\vspace{-.3cm}

\bibitem{cw77}
R. R. Coifman and G. Weiss, Extensions of Hardy spaces and their use in analysis,
Bull. Amer. Math. Soc. 83 (1977), 569-645.

\vspace{-0.3cm}

\bibitem{bdd}
E.  Belinsky, F. Dai and Z.  Ditzian,  Multivariate approximating averages,
J. Approx. Theory  125 (2003), 85-105.

\vspace{-0.3cm}

\bibitem{dd04}
F. Dai and Z. Ditzian, Combinations of multivariate averages, J. Approx. Theory 131 (2004), 268-283.

\vspace{-0.3cm}

\bibitem{dgyy} F. Dai, A. Gogatishvili, D. Yang and W. Yuan,
Characterizations of Sobolev spaces via averages on balls,
Submitted.

\vspace{-0.3cm}

\bibitem{di93} Z. Ditzian and K. G. Ivanov, Strong converse inequalities, J. Anal. Math. 61(1993), 61-111.

\vspace{-0.3cm}

\bibitem{dp}
Z.  Ditzian and  A.  Prymak,
Approximation by dilated averages and $K$-functionals,
Canad. J. Math. 62 (2010), 737-757.

\vspace{-0.3cm}

\bibitem{dr99} Z. Ditzian and K. Runovskii, Averages and
$K$-functionals related to the Laplacian, J. Approx. Theory 97 (1999), 113-139.

\vspace{-0.3cm}

\bibitem{fs71} C. Fefferman and E. M. Stein, Some maximal inequalities,
Amer. J. Math. 93 (1971), 107-115.

\vspace{-0.3cm}

\bibitem{FJ90}
M. Frazier and B. Jawerth,  A discrete
transform and decompositions of distribution spaces,
J. Funct. Anal. 93 (1990),  34-170.

\vspace{-0.3cm}

\bibitem{fjw91}
M. Frazier, B. Jawerth and G. Weiss,
Littlewood-Paley Theory and the Study of Function Spaces,
CBMS Regional Conference Series in Mathematics, 79.
Published for the Conference Board of the Mathematical
Sciences, Washington, DC;
by the American Mathematical Society, Providence, RI, 1991.


\vspace{-0.3cm}

\bibitem{gkz}  A. Gogatishvili, P. Koskela and Y. Zhou,
Characterizations of Besov and Triebel-Lizorkin spaces
on metric measure spaces, Forum Math. 25 (2013),  787-819.


\vspace{-0.3cm}

\bibitem{ht11}
D. D. Haroske and H. Triebel, Embeddings of function spaces: a criterion in
terms of differences, Complex Var. Elliptic Equ. 56 (2011), 931-944.

\vspace{-0.3cm}

\bibitem{ht13}
D. D. Haroske and H. Triebel, Some recent developments in the theory of
function spaces involving differences,
J. Fixed Point Theory Appl. 13 (2013), 341-358.

\vspace{-0.3cm}

\bibitem{h96} P. Haj\l asz, Sobolev spaces on an arbitrary
metric spaces, Potential Anal. 5 (1996), 403-415.

\vspace{-0.3cm}
\bibitem{h03}
P. Haj\l asz, Sobolev spaces on metric-measure spaces, in: Heat kernels and analysis
on manifolds, graphs, and metric spaces (Paris, 2002), 173-218,
Contemp. Math., 338, Amer. Math. Soc., Providence, RI, 2003.

\vspace{-0.3cm}

\bibitem{hmy}
Y. Han, D. M\"uller and D. Yang, A theory of Besov and Triebel-Lizorkin spaces on metric measure spaces modeled on Carnot-Carath¨¦odory spaces, Abstr. Appl. Anal. 2008, Art. ID 893409, 250 pp.

\vspace{-0.3cm}

\bibitem{my} D. M\"uller and D. Yang, A difference characterization of
Besov and Triebel-Lizorkin spaces on RD-spaces, Forum Math. 21
(2009), 259-298.

\vspace{-0.3cm}

\bibitem{rs}
T. Runst and W. Sickel, Sobolev Spaces of Fractional Order, Nemytskij Operators,
and Nonlinear Partial Differential Equations, de Gruyter Series in
Nonlinear Analysis and Applications 3, Walter de Gruyter \& Co., Berlin, 1996

\vspace{-0.3cm}

\bibitem{sh00} N. Shanmugalingam, Newtonian spaces: An extension of
Sobolev spaces to metric measure spaces, Rev. Mat. Iberoamericana
16 (2000), 243-279.

\vspace{-0.3cm}

\bibitem{st}
W. Sickel and H. Triebel,
H\"older inequalities and sharp embeddings in function spaces
of $B^s_{pq}$ and $F^s_{pq}$ type,
Z. Anal. Anwendungen 14 (1995), 105-140.

\vspace{-0.3cm}

\bibitem{s71}
E. M. Stein,  Singular Integrals and Differentiability Properties of Functions,
Princeton Mathematical Series 30, Princeton University Press, Princeton, N.J., 1970.

\vspace{-0.3cm}

\bibitem{s93} E. M. Stein, Harmonic Analysis: Real-Variable Methods, Orthogonality,
and Oscillatory Integrals, Princeton Mathematical Series 43, Monographs in Harmonic Analysis III, Princeton University Press, Princeton, NJ, 1993.

\vspace{-0.3cm}

\bibitem{t83} H. Triebel, Theory of Function Spaces,
Birkh\"{a}user Verlag, Basel, 1983.

\vspace{-0.3cm}

\bibitem{t92}
H. Triebel, Theory of Function Spaces. II, Birkh\"auser Verlag, Basel, 1992.

\vspace{-0.3cm}

\bibitem{t10} H. Triebel, Sobolev-Besov spaces of measurable functions, Studia Math.
201 (2010), 69-86.

\vspace{-0.3cm}

\bibitem{t11} H. Triebel, Limits of Besov norms,
Arch. Math. (Basel) 96 (2011), 169-175.

\vspace{-0.3cm}

\bibitem{w69}
R. L. Wheeden, Lebesgue and Lipschitz spaces and
integrals of the Marcinkiewicz type, Studia Math. 32 (1969), 73-93.

\vspace{-0.3cm}

\bibitem{w72}  R. L. Wheeden,
A note on a generalized hypersingular integral,  Studia Math. 44 (1972), 17-26.

\vspace{-0.3cm}

\bibitem{y03} D. Yang, New characterizations of Haj\l asz-Sobolev
spaces on metric spaces, Sci. China Ser. A 46 (2003),
675-689.

\vspace{-0.3cm}

\bibitem{yyz} D. Yang, W. Yuan and Y. Zhou,
A new characterization of Triebel-Lizorkin spaces on
$\rr^n$, Publ. Mat. 57 (2013), 57-82.

\vspace{-0.3cm}

\bibitem{yz11}
D. Yang and Y. Zhou, New properties of Besov and Triebel-Lizorkin spaces
on RD-spaces, Manuscripta Math. 134 (2011), 59-90.

\vspace{-0.3cm}

\bibitem{ysy} W. Yuan, W. Sickel and D. Yang,
Morrey and Campanato Meet Besov, Lizorkin and Triebel, Lecture Notes
in Mathematics 2005, Springer-Verlag, Berlin, 2010, xi+281 pp.

\end{thebibliography}
\end{document}